\documentclass[11pt,a4paper]{article}
\usepackage{amsmath}
\usepackage{amsthm}
\usepackage{amssymb}
\usepackage{mathtools}
\usepackage{geometry}
\usepackage{hyperref}
\usepackage{bm}

\hypersetup{
    linktoc=page,
    linkcolor=red,          
    citecolor=blue,   
    colorlinks=true
}

\theoremstyle{plain}
\newtheorem{theorem}{Theorem}[section]
\newtheorem{proposition}[theorem]{Proposition}
\newtheorem{lemma}[theorem]{Lemma}
\newtheorem{corollary}[theorem]{Corollary}

\theoremstyle{definition}

\newtheorem{remark}[theorem]{Remark}

\newcommand{\N}{\mathbb{N}}
\newcommand{\Z}{\mathbb{Z}}

\newcommand{\Pro}{\mathbb{P}}

\newcommand{\calG}{\mathcal{G}}
\newcommand{\calV}{\mathcal{V}}
\newcommand{\calE}{\mathcal{E}}
\newcommand{\dist}{\textup{dist}}

\newcommand{\pp}{\textbf{\textup{p}}}
\newcommand{\ee}{\textbf{\textup{e}}}
\newcommand{\phph}{\textbf{\textup{q}}}

\author{Hugo Vanneuville\thanks{CNRS and Institut Fourier (Université Grenoble Alpes)}}

\title{Sharpness of Bernoulli percolation via couplings}

\date{}

\begin{document}

\maketitle

\abstract{In this paper, we consider Bernoulli percolation on a locally finite, transitive and infinite graph (e.g.\ the hypercubic lattice $\Z^d$). We prove the following estimate, where $\theta_n(p)$ is the probability that there is a path of $p$-open edges from $0$ to the sphere of radius~$n$:
\[
\forall p\in [0,1],\forall m,n \ge 1, \quad \theta_{2n} (p-2\theta_m(p))\le C\frac{\theta_n(p)}{2^{n/m}}.
\]
This result implies that $\theta_n(p)$ decays exponentially fast in the subcritical phase. It also implies the mean-field lower bound in the supercritical phase. We thus provide a new proof of the sharpness of the phase transition for Bernoulli percolation.

Contrary to the previous proofs of sharpness, we do not rely on any differential formula. The main novelty is a stochastic domination result which is inspired by [Russo, 1982].

We also discuss a consequence of our result for percolation in high dimensions, where it can be seen as a near-critical sharpness estimate.}

\section{Introduction}

\subsection{Sharpness of the phase transition}

In this paper, we consider Bernoulli bond percolation on a locally finite, vertex-transitive and countably infinite graph $\calG=(\calV,\calE)$ (e.g.\ the hypercubic lattice $\Z^d$). This model is defined as follows: each edge $e \in \calE$ is open with some probability $p \in [0,1]$ and closed otherwise, independently of the other edges. We let $\Pro_p$ denote the corresponding product probability measure on $\{0,1\}^{\calE}$, where $1$ means open and $0$ means closed.

\smallskip

We fix a vertex $0 \in \calV$ once and for all and we let $\dist : \calV^2 \rightarrow \N$ denote the graph distance on $\calG$. Also, for every $n \ge 1$, we let $S_n \subset B_n \subset \calV$ denote the sphere and ball of radius $n$ around $0$, i.e.\footnote{Here and in all the paper, $n$ is an integer. $m,i,j,k$ will also always be integers.}
\[
S_n = \{ v \in \calV : \dist(0,v)=n \} \quad \text{and} \quad B_n = \{ v \in \calV : \dist(0,v) \le n \}.
\]
We let $\calE_n$ be the set of edges in $B_n$, i.e.
\[
\calE_n = \{ \{v,w\} \in\calE : v,w \in B_n \}.
\]
If $W \subset \calV$ and $v \in \calV$, we let $\{ v \leftrightarrow W \}$ denote the event that there is a path of open edges from $v$ to some site of $W$ and we let $\{ v \nleftrightarrow W \} = \{ v \leftrightarrow W \}^c$. We also let
\[
\theta_n(p)=\Pro_p[0 \leftrightarrow S_n] \quad \text{and} \quad \theta(p) = \lim_{n \rightarrow +\infty} \theta_n(p).
\]
(And we use the convention that $\theta_n(p)=0$ for every $n \ge 1$ if $p<0$.) The critical point is
\[
p_c = \inf \{ p \in [0,1] : \theta(p)>0\}.
\]

We refer to \cite{Gri99,BR06} for more about percolation. Our main result is the following estimate.

\begin{theorem}\label{thm:quant}
Let $C_0=4\log2$. Then, for every $p \in [0,1]$ and $m,n \ge 1$,
\[
\theta_{2n} (p-2\theta_m(p)) \le C_0\frac{\theta_n(p)}{2^{n/m}}.
\]
\end{theorem}

\begin{remark}
One can actually replace $p-2\theta_m(p)$ by $p(1-2\theta_m(p))$ (cf.\ the proof at the beginning of Section \ref{sec:exp}). This can be useful when $p_c$ is small. (But we have chosen to minimize the number of parentheses in the statement.)
\end{remark}

\begin{remark}
If $\theta(p_c)=0$ then Theorem \ref{thm:quant} applied to $p=p_c$ gives some ``slightly subcritical'' information. In the case of percolation in high dimensions, Theorem~\ref{thm:quant} applied to $p=p_c$ is sharp in some sense. See Section \ref{sec:cons}, where we discuss this further and give a new proof of a near-critical sharpness bound from the recent papers \cite{CHS21,HMS21}.
\end{remark}

Theorem \ref{thm:quant} implies the following theorem that is often referred to as sharpness of the phase transition (or more precisely subcritical sharpness and the mean-field lower bound). The first item -- i.e.\ subcritical sharpness -- has originally been proven independently in \cite{Men86} and \cite{AB87} and the second item -- i.e.\ the mean-field lower bound -- has originally been proven in \cite{CC87}. We refer to \cite{DT16,DRT19} for more recent proofs.

\begin{theorem}\label{thm:sharpness}
Let $p \in [0,1]$.
\begin{itemize}
\item If $p<p_c$, then there exist $c,C>0$ such that, for every $n \ge 1$, $\theta_n(p)\le Ce^{-cn}$;
\item If $p>p_c$, then $\theta(p) \ge (p-p_c)/2$.
\end{itemize}
\end{theorem}

\begin{proof}
Let us first prove exponential decay in the subcritical phase: Let $p < p_c$. We observe that, for every $p' \in (p,p_c)$, there exists $m \ge 1$ such that $p \le p'-2\theta_m(p')$. We fix such $(p',m)$. Theorem \ref{thm:quant} applied to $p'$ (with the crude bound $\theta_n(p')\le 1$) implies that
\[
\forall n \ge 1, \quad \theta_{2n}(p) \le C_0\exp\Big(-\frac{\log2}{m}n\Big).
\]
This implies the result for $n$ even, so also for $n$ odd since $\theta_n(p)$ is decreasing in $n$.

\medskip

Let us now prove the mean field lower bound: Let $p \in [0,1]$ and $m \ge 1$. Theorem~\ref{thm:quant} implies that $\theta(p-2\theta_m(p))=0$. As a result, $p-2\theta_m(p) \le p_c$. If we let $m$ go to $+\infty$, we obtain that $\theta(p) \ge (p-p_c)/2$.
\end{proof}

\begin{remark}\label{rk:cc}
The following result was proven in \cite{CC87}: $\theta(p) \ge (p-p_c)/p$. We will also obtain this result, see the discussion below Corollary~\ref{cor:homog}. However, we will not obtain the bound $\theta(p) \ge \frac{p-p_c}{p(1-p_c)}$ proven in \cite{DT17}.
\end{remark}

\subsection{A stochastic domination result}

The main novelty of this paper is a stochastic domination result. The quantity that will appear in this result is in some sense more $1-\theta_n$ than $\theta_n$ itself, so we let
\[
\pi_n(p) = \Pro_p [ 0 \nleftrightarrow S_n ] = 1 - \theta_n(p) \quad \text{and} \quad \pi(p) = \lim_{n \rightarrow +\infty} \pi_n(p) = 1-\theta(p).
\]

Before stating our result, let us state a corollary. We let $\preceq$ denote stochastic domination between probability measures (i.e.\ the existence of an increasing coupling). We recall that $\calE_n=\{\{v,w\} \in \calE : v,w \in B_n \}$.

\begin{corollary}\label{cor:homog}
For any $m,n \ge 1$ and $p \in (0,1)$,
\[
\Pro_{p\pi_m(p)} \big[ \omega_{|\calE_n} \in \cdot \big] \preceq \Pro_p \big[ \omega_{|\calE_n} \in \cdot \; \big| \; 0 \nleftrightarrow S_{n+m} \big].
\]
By applying this to the decreasing event $\{0 \nleftrightarrow S_n\}$, we obtain that
\begin{equation}\label{eq:cor}
\pi_n(p\pi_m(p)) \ge \frac{\pi_n(p)}{\pi_{n+m}(p)}.
\end{equation}
\end{corollary}

One can interpret Corollary \ref{cor:homog} as follows in the case $n \gg m \gg 1$: up to a small error (that corresponds to what happens in the thin annulus $B_{n+m} \setminus B_n$), multiplying $p$ by $\pi_m(p)$ (which is very close to $1$ if $p<p_c$) has more effect than conditioning on $\{0 \nleftrightarrow S_{n+m}\}$.

\medskip

Inequality \eqref{eq:cor} implies the inequality $\theta(p) \ge (p-p_c)/p$ from Remark \ref{rk:cc}:

\begin{proof}
Let $p \in (p_c,1)$ and $m \ge 1$. Inequality~\eqref{eq:cor} implies that $\pi(p\pi_m(p))=1$. Therefore, $p\pi_m(p) \le p_c$. If we let $m$ go to $+\infty$, we obtain that $\pi(p)\le p_c/p$, or equivalently that $\theta(p)\ge (p-p_c)/p$.
\end{proof}

When the reader looked at Inequality \eqref{eq:cor}, they may have wanted to make a telescopic product. So let us do this right now.

\begin{lemma}\label{lem:prod}
For any $m \ge 1$ and $p \in (0,1)$,
\[
\prod_{i \ge 1} \pi_{im}(p\pi_m(p)) \ge \pi_m(p).
\]
\end{lemma}

\begin{proof}
By \eqref{eq:cor},
\[
\prod_{i \ge 1} \pi_{im}(p\pi_m(p)) \ge \prod_{i \ge 1} \frac{\pi_{im}(p)}{\pi_{(i+1)m}(p)} = \frac{\pi_m(p)}{\pi(p)} \ge \pi_m(p). \qedhere
\]
\end{proof}

\begin{remark}
Even if we will use Lemma \ref{lem:prod} rather than the present remark, let us note that a similar telescopic product already implies an interesting sharpness-type result. More precisely, let us prove in this remark that
\begin{equation}\label{eq:summ}
\forall p<p_c, \quad \sum_{n \ge 1} \theta_n(p)<+\infty.
\end{equation}
\begin{proof}[Proof of \eqref{eq:summ}]
Let $p \in (0,p_c)$. We first note that \eqref{eq:summ} is equivalent to $\prod_{n \ge 1} \pi_n(p)>0$, so let us prove this. Let $p' \in (p,p_c)$ and $m \ge 1$ such that $p \le p'\pi_m(p')$. We have
\[
\prod_{n \ge 1} \pi_n(p) \ge \prod_{n \ge 1} \pi_n(p'\pi_m(p')) \overset{\eqref{eq:cor}}{\ge} \prod_{n \ge 1} \frac{\pi_n(p')}{\pi_{n+m}(p')} = \prod_{n=1}^m \pi_n(p') > 0. \qedhere
\]
\end{proof}
\end{remark}

In order to prove Theorem \ref{thm:quant}, we need a result that is analogous to Corollary~\ref{cor:homog} but that involves inhomogeneous percolation. Let $\pp \in [0,1]^\calE$. In inhomogeneous percolation of parameter $\pp$, each edge $e$ is open with probability $\pp_e$ and closed otherwise, independently of the other edges. We let $\Pro_\pp$ denote the corresponding product probability measure on $\{0,1\}^\calE$ and we let
\[
\pi_n(\pp)=\Pro_\pp[0 \nleftrightarrow S_n].
\]
Moreover, for every $e=\{v,w\} \in \calE$ and $n \ge 1$, we choose to measure the probability that $e$ is not connected to $S_n$ by using the following quantity:
\[
\pi_n(e,\pp) = \min \{ \Pro_\pp [ v \nleftrightarrow S_n ] , \Pro_\pp [ w \nleftrightarrow S_n ] \}.
\]
For every $m,n \ge 1$, we define $\phph_n^m : (0,1)^\calE \rightarrow (0,1)^\calE$ by
\[
\phph_n^m(\pp)_e = 
\begin{cases}
\pp_e \pi_{n+m}(e,\pp) \text{ if $e \in \calE_n$,}\\
\text{$\pp_e$ otherwise}.
\end{cases}
\]
(The values of $\phph_n^m$ outside of $\calE_n$ have no importance in the present section, but this choice will be useful later.)

\medskip

Our more precise stochastic domination result is the following proposition.

\begin{proposition}\label{prop:coupl}
For any $m,n \ge 1$ and $\pp \in (0,1)^\calE$,
\[
\Pro_{\phph_n^m(\pp)} \big[ \omega_{|\calE_n} \in \cdot \big] \preceq \Pro_{\pp} \big[ \omega_{|\calE_n} \in \cdot \; \big| \; 0 \nleftrightarrow S_{n+m} \big].
\]
By applying this to the decreasing event $\{0 \nleftrightarrow S_n\}$, we obtain that
\begin{equation}\label{eq:prop}
\pi_n(\phph_n^m(\pp)) \ge \frac{\pi_n(\pp)}{\pi_{n+m}(\pp)}.
\end{equation}
\end{proposition}

Corollary \ref{cor:homog} is a consequence of this proposition:

\begin{proof}[Proof of Corollary \ref{cor:homog}]
Let $p \in (0,1)$ and define $\pp \in (0,1)^\calE$ by $\forall e \in \calE, \pp_e = p$. Let $e \in \calE_n$. Then, each endpoint of $e$ is at distance at least $m$ from $S_{n+m}$, so $\pi_{n+m}(e,\pp) \ge \pi_m(p)$. As a result, $\phph_n^m(\pp)_e \ge p\pi_m(p)$.
\end{proof}

\subsection{Further remarks}

\begin{itemize}
\item Contrary to the other proofs of sharpness (see \cite{Men86,AB87,DT16,DRT19}), we do not use any differential formula (such as Russo's formula).\footnote{In this spirit, it seems worth mentioning here the recent paper \cite{DMT21} that proposes a new proof -- without differential formulas -- of some scaling relations in the plane.} One could argue that this should make the proof easier to extend to some models with dependence. Indeed, the link between the derivative of the probabilities and the pivotal set can be very complicated in dependent models. Unfortunately (and contrary to \cite{DRT19}), we have not succeeded in extending the present proof to other models (such as FK percolation or some continuous percolation model). Moreover, our proof of sharpness is short but longer than the one of \cite{DT16} (see also \cite{DT17}). However, we believe that our approach (and more generally the fact that sharpness can be proven via couplings), intermediate results such as Proposition \ref{prop:coupl}, and our main result Theorem \ref{thm:quant}, may be useful. (See Section \ref{sec:cons} for a consequence of Theorem~\ref{thm:quant}.)

\item The proof of Proposition \ref{prop:coupl} is inspired by \cite[Section 3]{Rus82}. The main difference is that (even if we will not say this explicitly in order to simplify the exposition) we have used a decision tree. Decision trees are central objects in computer science that have been shown to be very useful to prove sharpness results in \cite{DRT19}. See \cite{GS14} for more about the application of decision trees to percolation.
\end{itemize}

\paragraph{Organization of the paper.} Theorem \ref{thm:quant} is proven in Section \ref{sec:exp} by using Lemma \ref{lem:prod} and~\eqref{eq:prop}. Proposition \ref{prop:coupl} is proven in Section \ref{sec:coupl}. In Section \ref{sec:cons}, we discuss a consequence of Theorem \ref{thm:quant} for percolation in high dimensions.

\paragraph{Acknowledgments.} I would like to thank Barbara Dembin and Tom Hutchcroft for inspiring discussions. Moreover, I would like to thank Vincent Beffara and Stephen Muirhead for a simplification of the proof of some lemmas. I am also grateful to Stephen Muirhead for his comments on a previous version of the paper.

\section{Proof of Theorem \ref{thm:quant}}\label{sec:exp}

Theorem \ref{thm:quant} essentially follows by applying \eqref{eq:prop} inductively and by using Lemma \ref{lem:prod} to say that, in this procedure, $p$ has not decreased too much. We first note that it is sufficient to prove the following:

\medskip

For every $p \in (0,1)$ and $i,m,n\ge 1$,
\begin{equation}\label{eq:gamma_suff}
\pi_{n+im}(p\pi_m(p)^2) \ge \pi_{n}(p)^{1/2^i}.
\end{equation}
\begin{proof}[Proof of Theorem \ref{thm:quant} by using \eqref{eq:gamma_suff}]
Let $p \in (0,1)$ and $m,n \ge 1$. We observe that Theorem~\ref{thm:quant} is obvious if $\theta_m(p)\ge 1/2$ or $m \ge n$, so we assume that $\theta_m(p)<1/2$ and $n >m$. Next, we note that
\[
p\pi_m(p)^2 \ge p-2\theta_m(p).
\]
As a result, \eqref{eq:gamma_suff} applied to $i=i_0:=\lfloor n/m \rfloor$ implies that
\[
\theta_{2n}(p-2\theta_m(p)) \le \theta_{n+i_0m}(p-2\theta_m(p)) \le 1 - (1-\theta_n(p))^{1/2^{i_0}}.
\]
We conclude by using that $1-(1-x)^a \le (2\log2)ax$ for every $x \in [0,1/2]$ and $a \ge 0$ and that $i_0 \ge n/m-1$.
\end{proof}

So let us prove \eqref{eq:gamma_suff}. Let $\pp \in (0,1)^\calE$ and $m,n \ge 1$. We first compare $\pi_{n+m}(\phph_n^m(\pp))$ with $\pi_n(\pp)$. To this purpose, we observe that $\pi_{n+m}(\phph_n^m(\pp))$ is larger than or equal to both $\pi_{n+m}(\pp)$ and $\pi_n(\phph_n^m(\pp))$. As a result,~\eqref{eq:prop} implies that
\begin{equation}\label{eq:sqrt_sigma}
\pi_{n+m}(\phph_n^m(\pp)) \ge \sqrt{\pi_n(\pp)}.
\end{equation}
For every $i,m,n \ge 1$, we let
\[
\phph_n^{m,(i)}=\phph_{n+(i-1)m}^m\circ \dots \circ \phph_{n+m}^m \circ \phph_n^m.
\]
By applying \eqref{eq:sqrt_sigma} inductively, we obtain that, for every $\pp \in (0,1)^\calE$ and $i,m,n \ge 1$,
\begin{equation}\label{eq:ind_sigma}
\pi_{n+im}(\phph_n^{m,(i)}(\pp)) \ge \pi_n(\pp)^{1/2^i}.
\end{equation}
In order to apply \eqref{eq:ind_sigma}, we need a lower bound on $\phph_n^{m,(i)}(\pp)$.

\begin{lemma}\label{lem:end2}
Let $p \in (0,1)$, define $\pp \in (0,1)^\calE$ by $\pp_e = p$ for every $e \in \calE$ and let $i,m,n \ge 1$. Then,
\[
\forall e \in \calE, \quad \phph_n^{m,(i)}(\pp)_e \ge p \prod_{j=1}^i \pi_{jm}(p).
\]
\end{lemma}

\begin{proof}
\noindent\textbf{First step: $e \in \calE_n$.} Let $p,i,m,n$ as in the statement. We observe that
\[
\phph_n^{m,(i)}=\phph_{n+(i-1)m}^m \circ \phph_n^{m,(i-1)}
\]
(where $\phph_n^{m,(0)}:=id$). In this step, we consider some $e \in \calE_n$. Since $e \in \calE_{n+(i-1)m}$, we have
\[
\phph_n^{m,(i)}(\pp)_e = \phph_n^{m,(i-1)}(\pp)_e \times \pi_{n+im}(e,\phph_n^{m,(i-1)}(\pp))\\
\ge \phph_n^{m,(i-1)}(\pp)_e \times \pi_{im}(p),
\]
where in the inequality we have used that $\phph_n^{m,(i-1)}(\pp) \le p$ and that the endpoints of $e$ are at distance at least $im$ from $S_{n+im}$. So the result  in the case $e \in \calE_n$ follows by induction on $i$.

\medskip

\noindent \textbf{Second step: $e \notin \calE_n$.} Let $p,i,m,n$ as in the statement. In this step, we consider some $e \notin \calE_n$. If $e \notin \calE_{n+(i-1)m}$, then $\phph_n^{m,(i)}(\pp)_e=\pp_e=p$. If $e \in \calE_{n+(i-1)m} \setminus \calE_n$, then we let $j \in \{1,\dots,i-1\}$ be the number such that $e \in \calE_{n+jm} \setminus \calE_{n+(j-1)m}$ and we observe that
\[
\phph_n^{m,(i)}(\pp)_e = \phph_{n+(i-1)m}^m \circ \dots \circ \phph_{n+jm}^m(\pp)_e = \phph_{n+jm}^{m,(i-j)}(\pp)_e.
\]
We conclude by applying the first step to $i-j$ and $n+jm$ instead of $i$ and $n$.
\end{proof}

We are now in shape to prove \eqref{eq:gamma_suff}.

\begin{proof}[Proof of \eqref{eq:gamma_suff}] Let $i,m,n \ge 1$ and $p \in (0,1)$, and define $\pp_m \in (0,1)^\calE$ by $(\pp_m)_e=p\pi_m(p)$ for every $e \in \calE$. By Lemma~\ref{lem:end2} applied to $p=p\pi_m(p)$ and then by Lemma~\ref{lem:prod},
\[
\forall e \in \calE, \quad \phph_n^{m,(i)}(\pp_m)_e \ge p\pi_m(p) \prod_{j \ge 1} \pi_{jm}(p\pi_m(p)) \ge p\pi_m(p)^2.
\]
This together with \eqref{eq:ind_sigma} applied to $\pp=\pp_m$ implies that
\[
\pi_{n+im}(p\pi_m(p)^2) \ge \pi_n(p\pi_m(p))^{1/2^i} \ge \pi_n(p)^{1/2^i}. \qedhere
\]
\end{proof}

\section{Proof of Proposition \ref{prop:coupl}}\label{sec:coupl}

Let $m,n \ge 1$ and $\pp \in (0,1)^\calE$. In order to simplify the notations, we write
\[
A = \{ 0 \nleftrightarrow S_{n+m} \}, \quad E=\calE_{n+m}, \quad N=\text{Card}(E) \quad \text{and} \quad F=\calE_n.
\]
In this section, we prove Proposition \ref{prop:coupl}. To this purpose, we construct a coupling $(X,Y)=(X_e,Y_e)_{e \in E}$ of
\[
\mu_\pp := \Pro_{\pp} \big[ \omega_{|E} \in \cdot \; \big| \; A \big] \quad \text{and} \quad \Pro_{\phph_n^m(\pp)} \big[ \omega_{|E} \in \cdot \big]
\]
such that $X_{|F} \ge Y_{|F}$. (It would have been sufficient to construct the coupling in $F$ but it will be easier to construct it in $E$.)

\pagebreak

Let $(U_e)_{e \in E}$ be i.i.d.\ uniform random variables in $[0,1]$. First, we define $Y$ by $Y_e = 1$ if $U_e\le\phph_n^m(\pp)_e$ and $Y_e=0$ otherwise, for every $e \in E$.

\medskip

We use notations similar to those in \cite{DRT19}: we let $\vec{E}$ denote the set of all $N$-uples $e=(e_1,\dots,e_N) \in E^N$ with $e_j \ne e_k$ for all $j \ne k$ (i.e.\ this is the set of all possible orderings of $E$). We fix some $e^0 \in \vec{E}$ that we will just use to make some arbitrary decisions: when we refer to the ``first edge'', we always refer to the ordering given by $e^0$.

We define $X$ together with a random variable $\ee$ with values in $\vec{E}$ inductively as follows, where we use the notations $e_{[k]}=(e_1,\dots,e_k)$ and $x_{e_{[k]}}=(x_{e_1},\dots,x_{e_k})$:
\begin{itemize}
\item We let $\ee_1$ be the first edge adjacent to $0$;
\item We let $X_{\ee_1} = 1$ if $U_{\ee_1} \le \mu_\pp [ \omega_{\ee_1} = 1 ]$ and $X_{\ee_1} = 0$ otherwise;
\item Let us assume that $\ee_{[k-1]}$ and $X_{\ee_{[k-1]}}$ have been defined. If there exists an edge $\{v,w\} \in E \setminus \{\ee_1,\dots, \ee_{k-1}\}$ such that $v$ or $w$ is connected to $0$ by edges in $\{\ee_1,\dots,\ee_{k-1}\}$ that are open in $X$, then we let $\ee_k$ be the first such edge. Otherwise, we just let $\ee_k$ be the first edge in $E \setminus \{\ee_1,\dots, \ee_{k-1}\}$;
\item We let $X_{\ee_k} = 1$ if $U_{\ee_k} \le \mu_\pp [\omega_{\ee_k} = 1 \mid \omega_{\ee_{[k-1]}} = X_{\ee_{[k-1]}} ]$ and $X_{\ee_k} = 0$ otherwise.
\end{itemize}
Let us show that $X$ has the desired law (see \cite[Lemma 2.1]{DRT19} for a more general result -- we include a proof here for completeness).

\begin{lemma}\label{lem:law}
In the construction above, we have $X \sim \mu_\pp$.
\end{lemma}

\begin{proof}
Let $x$ in the support of $X$ (that is included in that of $\mu_\pp$ by construction) and let $e$ be the unique element of $\vec{E}$ such that $\ee=e$ on $\{X=x\}$. We have
\[
\Pro [ X = x ] = \prod_{k=1}^N \Pro \big[ X_{e_k} = x_{e_k} \; \big| \; X_{e_{[k-1]}} = x_{e_{[k-1]}} \big].
\]
By construction of $(X,\ee)$, $\{ X_{e_{[k-1]}} = x_{e_{[k-1]}} \} = \{ X_{e_{[k-1]}} = x_{e_{[k-1]}} \} \cap \{ \ee_{[k]} = e_{[k]} \}$. As a result, the above equals
\begin{align*}
&\prod_{k=1}^N
\begin{cases}
\Pro \big[ U_{e_k} \le \mu_\pp \big[ \omega_{e_k} = 1 \; \big| \; \omega_{e_{[k-1]}} = X_{e_{[k-1]}} \big] \; \big| \; X_{e_{[k-1]}} = x_{e_{[k-1]}} \big] \text{ if $x_{e_k}=1$}\\
\Pro \big[ U_{e_k} > \mu_\pp \big[ \omega_{e_k} = 1 \; \big| \; \omega_{e_{[k-1]}} = X_{e_{[k-1]}} \big] \; \big| \; X_{e_{[k-1]}} = x_{e_{[k-1]}} \big] \text{ if $x_{e_k}=0$}
\end{cases}
\\
& = \prod_{k=1}^N \mu_\pp \big[ \omega_{e_k} = x_{e_k} \; \big| \; \omega_{e_{[k-1]}} = x_{e_{[k-1]}} \big] = \mu_\pp[x],
\end{align*}
where we have used that $U_{e_k}$ is independent of $X_{e_{[k-1]}}$ (because the latter is measurable with respect to $U_{e_{[k-1]}}$).
\end{proof}

We are now in shape to prove Proposition \ref{prop:coupl}.
\begin{proof}[Proof of Proposition \ref{prop:coupl}]
It is sufficient to show that $X_{|F} \ge Y_{|F}$. As a result, it is sufficient to show that, if $(x,e) \in \{0,1\}^E \times \vec{E}$ is in the support of $(X,\ee)$, then for every $k$ such that $e_k \in F=\calE_n$, we have
\[
\mu_\pp \big[ \omega_{e_k} = 1 \; \big| \; \omega_{e_{[k-1]}} = x_{e_{[k-1]}} \big] \ge \phph_n^m(\pp)_{e_k} = \pp_{e_k}\pi_{n+m}(e_k,\pp).
\]
Let us fix such $(x,e)$ and $k$. We recall that we say that an edge is pivotal for $A$ if changing the state of this edge changes $1_{A}$. We first observe that, under $\Pro_\pp$, $\omega_{e_k}$ is independent of $A \cap \{ \omega_{e_{[k-1]}}=x_{e_{[k-1]}} \} \cap \{ \text{$e_k$ is not piv.\ for $A$} \}$ (because the latter is measurable with respect to $\omega_{|E \setminus \{e_k\}})$. By using this in the second equality below, we obtain that
\begin{multline*}
\mu_\pp \big[\omega_{e_k} = 1 \; \big| \; \omega_{e_{[k-1]}}=x_{e_{[k-1]}} \big]\\
= \Pro_\pp \big[\omega_{e_k} = 1, \text{$e_k$ is not piv.\ for $A$} \; \big| \; A, \omega_{e_{[k-1]}}=x_{e_{[k-1]}} \big]\\
= \pp_{e_k} \Pro_\pp \big[\text{$e_k$ is not piv.\ for $A$} \; \big| \; A, \omega_{e_{[k-1]}}=x_{e_{[k-1]}} \big].
\end{multline*}
We deduce that it is sufficient to prove that
\begin{equation}\label{eq:suff_piv}
\Pro_\pp \big[\text{$e_k$ is not piv.\ for $A$} \; \big| \; A, \omega_{e_{[k-1]}}=x_{e_{[k-1]}} \big] \ge \pi_{n+m}(e_k,\pp). 
\end{equation}
In order to prove this, we distinguish between two cases, where $\tau$ is the last $j \in \{1,\dots,N\}$ at which at least one endpoint of $e_j$ is connected to $0$ by edges in $\{e_1,\dots,e_{j-1}\}$ that are open in $x$.
\begin{itemize}
\item If $k > \tau$, then knowing that $\omega_{e_{[k-1]}}=x_{e_{[k-1]}}$ is sufficient to know whether $A$ holds or not, so the left-hand-side of \eqref{eq:suff_piv} equals $1$;
\item If $k \le \tau$, then we can write $e_k=\{v_k,w_k\}$ where $w_k$ is connected to $0$ by edges in $\{e_1,\dots,e_{k-1}\}$ that are open in $x$. We observe that, if $e_k$ is pivotal for $A$ and if $\omega_{e_{[k-1]}}=x_{e_{[k-1]}}$, then there is an open path from $v_k$ to $S_{n+m}$ that does not use any edge from $\{e_1,\dots,e_k\}$. We let $B$ denote this event. By this observation and then the Harris--FKG inequality (see \cite[Theorem 2.4]{Gri99} or \cite[Lemma 3]{BR06}) applied to $\Pro_\pp [ \, \cdot \mid \omega_{e_{[k-1]}} = x_{e_{[k-1]}}]$, we have
\begin{align*}
&\Pro_\pp \big[ \text{$e_k$ is piv.\ for $A$} \; \big| \; A, \omega_{e_{[k-1]}} = x_{e_{[k-1]}} \big]\\
&\le \Pro_\pp \big[ B \; \big| \; A, \omega_{e_{[k-1]}} = x_{e_{[k-1]}} \big]\\
&\le \Pro_\pp \big[ B \; \big| \; \omega_{e_{[k-1]}}= x_{e_{[k-1]}} \big] \text{ (this is where we have used Harris--FKG)}\\
&= \Pro_\pp [ B ] \le \Pro_\pp [ v_k \leftrightarrow S_{n+m} ] \le 1-\pi_{n+m}(e_k,\pp).
\end{align*}
\end{itemize}
This shows \eqref{eq:suff_piv} and ends the proof.
\end{proof}

\section{A consequence for percolation in high dimensions}\label{sec:cons}

Consider bond percolation on the hypercubic lattice $\Z^d$ for some $d \ge 11$. The following near-critical sharpness estimate has recently been proven in \cite{CHS21,HMS21}.

\begin{theorem}[{\cite[Theorem 1]{CHS21} and \cite[Theorem 1.2]{HMS21}}]\label{thm:hms}
There exist $c,C,\varepsilon_0>0$ such that, for every $\varepsilon \in (0,\varepsilon_0)$ and $n \ge 1$,
\[
\frac{c}{n^2}\exp(-C\varepsilon^{1/2}n) \le \theta_n(p_c-\varepsilon) \le \frac{C}{n^2}\exp(-c\varepsilon^{1/2}n).
\]
\end{theorem}

In this subsection, we give a new (shorter) proof of the upper bound from Theorem \ref{thm:hms}. (The techniques of the present paper seem quite useless in order to prove the lower bound.) We need the following result from \cite{KN11} (that relies on \cite{HS94} for $d \ge 19$ and on \cite{FH17} for the extension to $d \ge 11$):

There exist $c,C>0$ such that, for every $n \ge 1$,
\begin{equation}\label{eq:KN11}
\frac{c}{n^2} \le \theta_n(p_c) \le \frac{C}{n^2}. 
\end{equation}
(Actually, we just need the upper bound.) 

\begin{proof}[Proof of the upper bound from Theorem \ref{thm:hms}]
We fix a constant $C$ such that \eqref{eq:KN11} holds. Also, we fix some $\varepsilon_0>0$ such that
\[
(2C/\varepsilon_0)^{1/2} \ge 1.
\]
Let $\varepsilon \in (0,\varepsilon_0)$ and write $m=\lceil (2C/\varepsilon)^{1/2} \rceil$. We have
\begin{multline*}
\forall n \ge 1, \quad \theta_n(p_c-\varepsilon) \overset{\eqref{eq:KN11}}{\le} \theta_n( p_c - 2\theta_m(p_c) )\overset{\text{\eqref{eq:KN11} and Thm.\ \ref{thm:quant}}}{\le} \frac{C_0C}{n^2} \times \frac{1}{2^{n/m}}\\
\le \frac{C_0 C}{n^2} \exp\Big( -\frac{\log2}{2 \times (2C)^{1/2}}\varepsilon^{1/2} n \Big),
\end{multline*}
where in the last inequality we have used that $m \le 2 \times (2C/\varepsilon)^{1/2}$ (because $\varepsilon \le \varepsilon_0$).
\end{proof}

\begin{remark}
The reason why Theorem \ref{thm:quant} is sharp (in some sense) for percolation in high dimensions is that the mean-field equality $\eta_1 \nu = 1$ holds for these graphs. Here, $\eta_1$ is the one-arm exponent and $\nu$ is the correlation length exponent -- see e.g.\ the recent work \cite{DM21} for the definition of $\eta_1$ and $\nu$, some inequalities involving these exponents, and further references. For percolation on $\Z^d$ with $d \ge 11$, it has been proven that these exponents exist and satisfy $\eta_1=1/\nu=2$ \cite{HS94,KN11,FH17}. Let us note that, as one can see for instance by using Theorem~\ref{thm:quant}, if $\eta_1$ and $\nu$ exist, then they satisfy $\eta_1 \nu \le 1$.
\end{remark}

\footnotesize

\bibliographystyle{alpha}
\bibliography{ref_new_proof_sharpness}

\end{document}